\documentclass{amsart}
\usepackage[latin9]{inputenc}
\usepackage{amsthm}
\usepackage{amstext}

\makeatletter
\numberwithin{equation}{section}
\numberwithin{figure}{section}

\usepackage{amsfonts}

\usepackage{amscd}\usepackage{thmdefs}

\input{tcilatex}
\theoremstyle{definition}
\theoremstyle{remark}
\numberwithin{equation}{section}

\input tcilatex

\makeatother

\begin{document}

\title[Kirillov models and integral structures]{Kirillov models and the Breuil-Schneider conjecture for $GL_{2}(F)$}

\author{Eran Assaf}

\author{David Kazhdan}

\author{Ehud de Shalit}

\address{Hebrew University, Jerusalem, Israel}

\email{deshalit@math.huji.ac.il}
\begin{abstract}
Let $F$ be a local field of characteristic 0. The Breuil-Schneider
conjecture for $GL_{2}(F)$ predicts which locally algebraic representations
of this group admit an integral structure. We extend the methods of
{[}K-dS12{]}, which treated smooth representations only, to prove
the conjecture for some locally algebraic representations as well. 
\end{abstract}
\maketitle

\section{Introduction}

\subsection{Background}

Let $F$ be a local field of characteristic 0 and residue characteristic
$p$, $\pi$ a fixed uniformizer of $F,$ and $q$ the cardinality
of its residue field $\mathcal{O}_{F}/\pi\mathcal{O}_{F}.$ Let $E$
be an algebraic closure of $F$.

Let $\mathbf{G}$ be a reductive group over $F$ and $G=\mathbf{G}(F).$
A \emph{locally algebraic} representation $(\rho,V_{\rho})$ of $G$
over $E$ is a representation of the type 
\begin{equation}
\rho=\tau\otimes\sigma
\end{equation}
where $(\tau,V_{\tau})$ is (the $E$-points of) a finite dimensional
rational representation of $\mathbf{G}$, and $(\sigma,V_{\sigma})$
is a smooth representation of $G$ over $E.$ An \emph{integral structure}
$V_{\rho}^{0}$ in $V_{\rho}$ is an $\mathcal{O}_{E}[G]$-submodule
which spans $V_{\rho}$ over $E,$ but does not contain any $E$-line.

If $\tau$ and $\sigma$ are irreducible then $\rho$ is irreducible
as well ({[}P01{]}, Theorem 1). In such a case, a non-zero $\mathcal{O}_{E}[G]$-submodule
$V_{\rho}^{0}$ of $V_{\rho}$ is an integral structure if and only
if it is properly contained in $V_{\rho}.$ Indeed, the union of all
$E$-lines in $V_{\rho}^{0},$ as well as the subspace of $V_{\rho}$
spanned by $V_{\rho}^{0}$ over $E,$ are both $E[G]$-submodules
of $V_{\rho}.$ If $0\subset V_{\rho}^{0}\subset V_{\rho}$ (both
inclusions being proper), the irreducibility of $\rho$ implies that
the first is $0$, and the second is $V_{\rho}.$

Two integral structures in $V_{\rho}$ are \emph{commensurable }if
each of them is contained in a scalar multiple of the other. In general,
$V_{\rho}$ need not contain an integral structure. When such an integral
structure exists, it need not be unique, even up to commensurability.
However, if $\rho$ is irreducible, and an integral structure does
exist, there is a unique commensurability class of \emph{minimal integral
structures}, namely the class of any cyclic $\mathcal{O}_{E}[G]$-module.
Thus, when $\rho$ is irreducible, to test whether integral structures
exist at all, it is enough to check that for some $0\neq v\in V_{\rho},$
$\mathcal{O}_{E}[G]v$ is not the whole of $V_{\rho}.$

The existence (and classification) of integral structures in irreducible
locally algebraic representations is a natural and important question
for the $p$\emph{-adic local Langlands programme }(see {[}Br10{]})\emph{.}
When $\mathbf{G}=GL_{n},$ a precise conjecture for the conditions
on $\tau$ and $\sigma$ under which an integral structure should
exist in $\rho$ was proposed by Breuil and Schneider in {[}Br-Sch07{]},
and became known as \emph{the Breuil-Schneider conjecture}. The \emph{necessity}
of these conditions was proved there in some special cases, and by
Hu {[}Hu09{]} in general. The \emph{sufficiency }tends to be, in the
words of Vigneras {[}V{]}, either ``obvious'' or ``very hard'',
even for $GL_{2}.$

Quite generally, if $\mathbf{G}$ is an arbitrary reductive group,
the simpler $\sigma$ is algebraicly, the harder the question becomes.
An obvious necessary condition is for the central character of $\rho$
to be unitary%
\footnote{A character $\chi:F^{\times}\rightarrow E^{\times}$ is \emph{unitary}
if its values lie in $\mathcal{O}_{E}^{\times}.$%
}. Assume therefore that this is the case. If $\sigma$ is supercuspidal
(its matrix coefficients are compactly supported modulo the center),
the existence of an integral structure is obvious. Using global methods
and the trace formula, existence of an integral structure can also
be proved when $\sigma$, realized over $\Bbb{C}$ by means of some
field embedding $E\hookrightarrow\Bbb{C}$, is essentially discrete
series (its matrix coefficients are square integrable modulo the center)%
\footnote{The notion of ``essentially discrete series'' should be invariant
under $Aut(\Bbb{C}),$ hence independent of the embedding of $E$
in $\Bbb{C}$. This is known for $GL_{n}$ by the work of Bernstein-Zelevinski,
and for the classical groups by Tadic.%
} {[}So13{]}. In these cases, no further restrictions are imposed on
$\tau$. At the other extreme stand principal series representations,
where one should impose severe restrictions on $\tau,$ and the problem
becomes very hard.

We warn the reader that for arithmetic applications, the minimal integral
structures in an irreducible $V_{\rho}$ are often insufficient. In
particular, they may be non-admissible, in the sense that their reduction
modulo the maximal ideal of $\mathcal{O}_{E}$ is a non-admissible
smooth representation over $\Bbb{\bar{F}}_{q}$. In such a case, even
if minimal integral structures are known to exist, the existence of
larger admissible integral structures is a mystery, which is resolved
only in special cases, again by global methods. See {[}Br04{]}.

\subsection{The main result}

We now specialize to $\mathbf{G}=GL_{2}.$ In this case the full Breuil-Schneider
conjecture is known when $F=\Bbb{Q}_{p},$ but only by indirect methods
involving $(\phi,\Gamma)$-modules and Galois representations. It
comes as a by-product of the proof of the $p$-adic local Langlands
correspondence (\emph{pLLC}). This large-scale project {[}B-B-C{]}
depends so far crucially on the assumption $F=\Bbb{Q}_{p}.$ It is
therefore desirable to have a \emph{direct local proof} of the Breuil-Schneider
conjecture, which does not depend on \emph{pLLC}, and which holds
for arbitrary $F$. As mentioned above, if $\sigma$ is either supercuspidal
or special, there are no restrictions on $\tau$ and integral structures
are known to exist. We therefore assume that $\sigma=Ind(\chi_{1},\chi_{2})$
is an irreducible principal series representation.

In this work we prove the Breuil-Schneider conjecture for $GL_{2}(F)$
in the following cases: (1) The characters $\chi_{1}$ and $\chi_{2}$
are unramified, $\tau=\det(.)^{m}\otimes Sym^{n},$ and the weight
is low: $n<q$ (2) The $\chi_{i}$ are tamely ramified, and $\tau=\det(.)^{m}.$
The second case has been done in {[}K-dS12{]} already, but the proof
presented here is somewhat cleaner.

To formulate our theorem, let $\chi_{i}$ be smooth characters of
$F^{\times}$ with values in $E^{\times},$ and $\omega$ the unramified
character%
\footnote{This character is usually denoted $|.|$ over $\Bbb{C}.$ We will
have to consider $|\omega(\pi)|$, the absolute value of $q^{-1}$
as an element of $E$, and we found the notation $||\pi||$ too confusing.%
} for which $\omega(\pi)=q^{-1}$. Let $B$ be the Borel subgroup of
upper triangular matrices in $G,$ and consider the principal series
representation 
\begin{equation}
(V_{\sigma},\sigma)=Ind_{B}^{G}(\chi_{1},\chi_{2}).
\end{equation}
This is the space of functions $f:G\rightarrow E$ for which (i) 
\begin{equation}
f\left(\left(\begin{array}{ll}
t_{1} & s\\
0 & t_{2}
\end{array}\right)g\right)=\chi_{1}(t_{1})\chi_{2}(t_{2})f(g)
\end{equation}
and (ii) there exists an open subgroup $H\subset G,$ depending on
$f,$ such that $f(gh)=f(g)$ for all $h\in H.$ The group $G$ acts
by right translation: 
\begin{equation}
\sigma(g)f(g^{\prime})=f(g^{\prime}g).
\end{equation}
The central character of $\sigma$ is $\chi_{1}\chi_{2}$, and $Ind_{B}^{G}(\chi_{1},\chi_{2})\simeq Ind_{B}^{G}(\omega\chi_{2},\omega^{-1}\chi_{1}),$
unless this representation is reducible. In fact, $\sigma$ is reducible
precisely when $\chi_{1}/\omega\chi_{2}=\omega^{\pm1}.$ In this ``special''
case $\sigma$ is indecomposable of length 2, and its irreducible
constituents are a one-dimensional character and a twist of the Steinberg
representation by a character. Since the Breuil-Schneider conjecture
for a twist of Steinberg, and any $\tau$, is known (for $GL_{2}(F),$
see {[}T93{]} or {[}V08{]}), we exclude this case from now on, and
assume that $\sigma$ is \emph{irreducible}.

Next, fix integers $m$ and $n\ge0,$ and consider the rational representation
\begin{equation}
(V_{\tau},\tau)=\det(.)^{m}\otimes Sym^{n},
\end{equation}
where $Sym^{n}$ denotes the $n$th symmetric power of the standard
representation of $\mathbf{G}$. Put 
\begin{eqnarray}
\lambda & = & \chi_{1}(\pi),\,\,\mu=\omega\chi_{2}(\pi),\\
\tilde{\lambda} & = & \lambda\pi^{m},\,\,\,\tilde{\mu}=\mu\pi^{m}.\notag
\end{eqnarray}
The Breuil-Schneider conjecture for $\rho=\tau\otimes\sigma$ predicts
that $\rho$ has an integral structure if and only if the following
two conditions are satisfied: 
\begin{equation}
\text{(i)\thinspace\thinspace}|\tilde{\lambda}\tilde{\mu}q\pi^{n}|=1\,\,\,\,\text{(ii) }|\tilde{\lambda}|\le|q^{-1}\pi^{-n}|,\,|\tilde{\mu}|\le|q^{-1}\pi^{-n}|.
\end{equation}
Condition (i) means that the central character of $\rho$ is unitary.
Given (i), (ii) is equivalent to $1\le|\tilde{\lambda}|\le|q^{-1}\pi^{-n}|$
or to the symmetric condition for $\tilde{\mu}.$ It is known (and
easy to prove) that these two conditions are necessary.

\begin{theorem} Assume that (i) and (ii) are satisfied. Assume, in
addition, that \emph{either} (1) $\chi_{1}$ and $\chi_{2}$ are unramified
and $n<q,$ \emph{or} (2) that $\chi_{1}$ and $\chi_{2}$ are tamely
ramified and $n=0.$ Then $\rho$ has an integral structure. \end{theorem}

Although our method is new, and gives some new insight into the minimal
integral structure (see Theorem 1.2 below), the two cases have been
known before: case (1) by Breuil {[}Br03{]} (for $\Bbb{Q}_{p}$) and
de Ieso {[}dI12{]} (for general $F$), and case (2) by Vigneras {[}V08{]}.
It is interesting to note that the restriction $n<q$ in case (1)
and the restriction on tame ramification in case (2) are also needed
in the above mentioned works. In fact, Breuil, de Ieso and Vigneras
all use, in one way or another, the method of \emph{compact induction},
replacing the representation $\rho$ by a local system on the tree
of $G$. Our approach takes place in a certain \emph{dual} space of
functions on $F.$ Any attempt to translate it to the set-up of the
tree involves the $p$-adic Fourier transform, which is unbounded,
and makes it impossible to trace back the arguments. The way in which
the weight and ramification restrictions are brought to bear on the
problem are also not similar, yet the very same restrictions turn
out to be necessary for the proofs to work.

\subsection{An outline of the proof}

As in {[}K-dS12{]}, our approach is based on a study of the Kirillov
model of $\rho.$ For the sake of exposition we now exclude the case
$\chi_{1}=\omega\chi_{2}$, which requires special attention. Assuming
$\chi_{1}\neq\omega\chi_{2},$ the \emph{Kirillov model} of $\rho$
is then the following space of functions on $F-\left\{ 0\right\} $:
\begin{equation}
\mathcal{K}=C_{c}^{\infty}(F,\tau)\chi_{1}+C_{c}^{\infty}(F,\tau)\omega\chi_{2}.
\end{equation}
Here $C_{c}^{\infty}(F,\tau)$ is the space of $V_{\tau}$-valued
locally constant functions of compact support on $F.$ The model $\mathcal{K}$
is obtained by tensoring $\tau$ with the classical Kirillov model
of the smooth representation $\sigma\ $(see {[}Bu98{]}). It contains
$\mathcal{K}_{0}=C_{c}^{\infty}(F^{\times},\tau),$ the subspace of
functions vanishing near 0, and $\mathcal{K}/\mathcal{K}_{0}$ consists
of two copies of $V_{\tau}.$ When $\tau=1,$ this is just the \emph{Jacquet
module} of $\mathcal{K}.$ The characters $\chi_{1}$ and $\omega\chi_{2}$
are the \emph{exponents} of the Jacquet module, the two characters
by which the torus of diagonal matrices acts on it.

We record the action of an element 
\begin{equation}
g=\left(\begin{array}{ll}
a & b\\
0 & 1
\end{array}\right)\in B
\end{equation}
on $\phi\in\mathcal{K}.$ Fix an additive character $\psi:F\rightarrow E^{\times}$
under which $\mathcal{O}_{F}$ is its own annihilator. Then 
\begin{equation}
\rho(g)\phi(x)=\tau(g)\left(\psi(bx)\phi(ax)\right).
\end{equation}
The action of $G$ in the model $\mathcal{K}$ depends on the choice
of $\psi,$ but only up to isomorphism.

At this point, we must introduce more notation and recall some easy
facts. Let $1_{S}$ be the characteristic function of $S\subset F,$
and $\phi_{l}=1_{\pi^{l}U_{F}}$ ($l\in\Bbb{Z}$). If $b\in F$, write
$\psi_{b}(x)=\psi(bx).$ The function $\psi_{b}(\pi^{-l}x)\phi_{l}(x)$
depends only on $\beta,$ the image of $b$ in $W=F/\mathcal{O}_{F}$,
so from now on we denote it by $\psi_{\beta}(\pi^{-l}x)\phi_{l}(x).$
Any locally constant function on the annulus $\pi^{l}U_{F}$ can be
expanded as a finite linear combination of these functions. Moreover,
Fourier analysis on the disk $\pi^{l}\mathcal{O}_{F}$ implies that
\begin{equation}
\sum_{\beta\in W}C_{l}(\beta)\psi_{\beta}(\pi^{-l}x)\phi_{l}(x)=0
\end{equation}
if and only if $C_{l}(\beta)$ depends only on $\pi\beta$, i.e. 
\begin{equation}
C_{l}(\beta)=C_{l}(\beta^{\prime})\text{ if }\beta-\beta^{\prime}\in W_{1}=\pi^{-1}\mathcal{O}_{F}/\mathcal{O}_{F}.
\end{equation}
The same applies of course to $V_{\tau}$-valued functions, except
that now the coefficients $C_{l}(\beta)\in V_{\tau}.$

An arbitrary function $\phi\in\mathcal{K}$ may be expanded annulus-by-annulus
as 
\begin{equation}
\phi=\sum_{l=l_{0}}^{\infty}\sum_{\beta\in W}C_{l}(\beta)\psi_{\beta}(\pi^{-l}x)\phi_{l}(x),
\end{equation}
where $C_{l}(\beta)\in V_{\tau},$ and for every $l$ only finitely
many $C_{l}(\beta)\neq0.$ The only restriction on $\phi$ is imposed
by the asymptotics as $x\rightarrow0.$ In particular, finite linear
combinations as above represent the elements of $\mathcal{K}_{0}$.
One should think of the $\beta$ as frequencies, and of the $C_{l}(\beta)$
as the amplitudes attached to these frequencies on the annulus $\pi^{l}U_{F}.$
These amplitues are not uniquely defined since we may add to $C_{l}(\beta)$
a perturbation $\tilde{C}_{l}(\beta)$ without affecting $\phi|\pi^{l}U_{F},$
provided $\tilde{C}_{l}(\beta)=\tilde{C}_{l}(\beta^{\prime})$ whenever
$\beta-\beta^{\prime}\in W_{1}.$ But as explained above, this is
the only ambiguity.

Theorem 1.1 follows from the following more precise result, which
makes the integral structure on $V_{\rho}$ ``visible''.

\begin{theorem} Let the assumptions be as in Theorem 1.1. Let $V_{\rho}^{0}$
be the $\mathcal{O}_{E}[G]$-submodule of $V_{\rho}=\mathcal{K}$
spanned by a non-zero vector. Then there exist $\mathcal{O}_{E}$-lattices
$M_{0}(\beta)\subset V_{\tau}$ such that if $\phi\in V_{\rho}^{0}$
vanishes outside $\mathcal{O}_{F},$ it has an expansion as above
with $C_{0}(\beta)\in M_{0}(\beta)$ for every $\beta.$ \end{theorem}

Note that we do not claim that the values of $\phi\in V_{\rho}^{0}$
are bounded on $U_{F},$ nor at any other point. The amplitudes can
be bounded only separately, and only on the first annulus where $\phi$
does not vanish. Since the $C_{0}(\beta)$ are not uniquely defined,
one still needs a simple argument to show that this is good enough.

\begin{proposition} Theorem 1.2 implies Theorem 1.1. \end{proposition}

\begin{proof} We shall show that $V_{\rho}^{0}\neq V_{\rho},$ so
in view of the irreducibility of $\rho,$ $V_{\rho}^{0}$ will be
an integral structure. Consider the function $\phi=C\phi_{0}$ where
$C\in V_{\tau}$ \emph{lies outside} $M=\sum_{\beta\in W_{1}}M_{0}(\beta).$
Suppose, by way of contradiction, that $\phi\in V_{\rho}^{0}.$ Then
$\phi$ is also given by an expansion as in Theorem 1.2. For $x\in U_{F}$
we must have then 
\begin{equation}
C=\sum_{\beta\in W}C_{0}(\beta)\psi_{\beta}(x).
\end{equation}
This forces, as we have seen, the equality $C_{0}(0)-C=C_{0}(\beta)$
for $\beta\in W_{1}-\left\{ 0\right\} .$ But this contradicts the
choice of $C$. \end{proof}

We now make some comments on the proof of Theorem 1.2. The first step
is standard. Using the decomposition $G=BK,$ $K=GL_{2}(\mathcal{O}_{F}),$
we show that $V_{\rho}^{0}$ is commensurable with a certain $\mathcal{O}_{E}[B]$-module
of finite type $\Lambda$ which also spans $V_{\rho}$ over $E$.
We may therefore prove the assertion of the theorem for $\Lambda$
instead of $V_{\rho}^{0}.$ Our $\Lambda$ will be spanned over $\mathcal{O}_{E}$
by an explicit infinite set $\mathcal{E}$ of nice functions.

Pick a $\phi\in\Lambda,$ express it as a linear combination of the
functions in $\mathcal{E}$, and expand it annulus-by-annulus as above.
The coefficients $C_{l}(\beta)$ then satisfy \emph{recursive relations},
in which the coefficients used to express $\phi$ as a linear combination
of $\mathcal{E}$ figure out.

Suppose that $\phi$ vanishes off $\mathcal{O}_{F}.$ It may still
be the case that $C_{l}(\beta)\neq0$ for some $\beta$ and $l<0.$
However, cancellation must take place, and as we have seen, $C_{l}(\beta)$
depends then, for $l<0,$ on $\pi\beta$ only. We proceed by increasing
induction on $l$ and show that $C_{l}(\beta)$ must belong, for $l\le0,$
to a certain $\mathcal{O}_{E}$-lattice $M_{l}(\beta)\subset V_{\tau},$
depending on $l$ and $\beta,$ but not on $\phi.$ When $l=0$ we
reach the desired conclusion.

Two phenomena assist us in establishing these bounds on the coefficients.
The first, which has already been utilized in our previous work {[}K-dS12{]},
is that in the recursive relations for $C_{l}(\beta)$ we encounter
terms such as 
\begin{equation}
\sum_{\pi\alpha=\beta}C_{l-1}(\alpha).
\end{equation}
As long as $l\le0,$ the $q$ summands are all equal, so their sum
is equal to $qC_{l-1}(\alpha_{\beta}),$ where $\alpha_{\beta}$ is
any one of the $\alpha$'s. The factor $q$ is small, and helps to
control $C_{l}(\beta).$

The second phenomenon is new, and more subtle. The information that
$C_{l}(\beta)$ depends only on $\pi\beta,$ puts a further restriction
on $C_{l}(\beta),$ beyond lying in $M_{l}(\beta),$ which is vital
for the deduction that the $C_{l+1}(\gamma)$ lie in $M_{l+1}(\gamma).$
For example, assume that $m=0$ and $n=1,$ so $\tau$ is the standard
representation of $G$ on $E^{2}$, and let $e_{1}$ and $e_{2}$
be the standard basis. In this example, up to scaling, 
\begin{equation}
M_{l}(\beta)=Span_{\mathcal{O}_{E}}\left\{ \pi^{-l}e_{1},e_{2}-\pi^{-l}\beta e_{1}\right\} 
\end{equation}
(note that this is indeed well defined, i.e. depends only on $\beta\func{mod}\mathcal{O}_{F}$).
It is easily checked that if $C_{l}(\beta)\in M_{l}(\beta)$ for all
$\beta,$ and \emph{in addition, }$C_{l}(\beta)$ depends only on
$\pi\beta,$ then in fact 
\begin{equation}
C_{l}(\beta)\in Span_{\mathcal{O}_{E}}\left\{ \pi^{-l}e_{1},\pi(e_{2}-\pi^{-l}\beta e_{1})\right\} .
\end{equation}
This minor improvement on $C_{l}(\beta)\in M_{l}$ is crucial for
our method to work. Roughly speaking, the first phenomenon described
above takes care of the factor $q^{-1}$ in condition (1.7)(ii), while
the second one takes care of the $\pi^{-n}.$

The inductive procedure requires also the relation $M_{l}(\beta)\subset M_{l+1}(\pi\beta).$
It is here that we need the condition $n<q.$ We may modify the definition
of $M_{l}(\beta)$ to guarantee this relation without any restriction
on $n,$ but we then lose the subtle phenomenon to which we alluded
in the previous paragraph. At present, we are unable to hold the rope
at both ends simultaneously.

When $\chi_{1}$ and $\chi_{2}$ are unramified this is the end of
the story. When $\chi_{1}$ and $\chi_{2}$ are ramified, two types
of complications occur. First, we must give up the algebraic part
$\tau$ (except for the benign twist by the determinant). Second,
in the recursive relations used to define $C_{l}(\beta),$ Gauss sums
intervene. These Gauss sums have denominators which are still under
control if the characters are only tamely ramified, but if the $\chi_{i}$
are wildly ramified, our method breaks down. It is interesting to
note that the well-known estimates on Gauss sums intervene also in
Vigneras' proof of the tamely-ramified smooth case of the conjecture.

In the remaining cases, not covered by (1) or (2), it is possible
that Theorem 1.2 fails, yet Theorem 1.1 continues to hold, for a different
reason. It will be interesting to check numerically whether one should
expect Theorem 1.2 in general. Even for $F=\Bbb{Q}_{p},$ where, as
mentioned above, the full conjecture is known, it is unclear to us
whether Theorem 1.2 holds beyond cases (1) and (2).

\section{Preliminary results}

\subsection{Fourier analysis on $\mathcal{O}_{F}$}

The discrete group $W=F/\mathcal{O}_{F}$ is the topological dual
of $\mathcal{O}_{F}$ via the pairing 
\begin{equation}
(\beta,x)\mapsto\psi_{\beta}(x)=\psi(\beta x).
\end{equation}
Every locally constant $E$-valued function on $\mathcal{O}_{F}$
has a unique finite Fourier expansion 
\begin{equation}
\phi=\sum_{\beta\in W}c(\beta)\psi_{\beta}(x).
\end{equation}
The proof of the following easy lemma is left to the reader.

\begin{lemma} (i) $\phi|U_{F}=0$ if and only if $c(\beta)$ depends
only on $\pi\beta.$

(ii) $\phi|\pi\mathcal{O}_{F}=0$ if and only if $\sum_{\pi\beta=\gamma}c(\beta)=0$
for every $\gamma\in W.$ \end{lemma}

The lemma is immediately translated to a similar one in the disk $\pi^{l}\mathcal{O}_{F}$
using the functions $\psi_{\beta}(\pi^{-l}x)$ as a basis for the
expansion.

\subsection{Lattices in $V_{\tau}$}

If $\beta\in W$ and $l\in\Bbb{Z}$ let 
\begin{equation}
D_{l}(\beta)=\left\{ u\in F|\,|u-\pi^{-l}\beta|\le|\pi^{-l}|\right\} .
\end{equation}
This disk indeed depends only on $\beta\func{mod}\mathcal{O}_{F}.$
Note that 
\begin{equation}
D_{l+1}(\gamma)=\coprod_{\pi\beta=\gamma}D_{l}(\beta).
\end{equation}

Let $\tau=\det(.)^{m}\otimes Sym^{n}.$ Identify $V_{\tau}$ with
$E[u]^{\le n},$ the space of polynomials of degree at most $n,$
with the action 
\begin{equation}
\tau\left(\left(\begin{array}{ll}
a & b\\
c & d
\end{array}\right)\right)u^{i}=(ad-bc)^{m}(a+cu)^{n-i}(b+du)^{i}.
\end{equation}

Let 
\begin{equation}
N_{l}(\beta)=\left\{ P\in V_{\tau}|\,|P(u)|\le|\pi|^{-nl}\,\,\forall u\in D_{l}(\beta)\right\} .
\end{equation}
These are lattices in $V_{\tau}.$

\begin{lemma} (i) For any $\gamma\in W$
\begin{equation}
\bigcap_{\pi\beta=\gamma}N_{l}(\beta)=\pi^{n}N_{l+1}(\gamma).
\end{equation}
(ii) Assume that $n<q.$ Then 
\begin{equation}
N_{l}(\beta)=Span_{\mathcal{O}_{E}}\left\{ (\pi^{-l})^{n-i}(u-\pi^{-l}\beta)^{i}\,\,\,(0\le i\le n)\right\} .
\end{equation}
(iii) Assume that $n<q.$ Then 
\begin{equation}
N_{l}(\beta)\subset N_{l+1}(\pi\beta).
\end{equation}
\end{lemma}

\begin{proof} (i) If $P\in N_{l}(\beta)$ then it is bounded by $|\pi|^{-nl}$
on $D_{l}(\beta).$ But the $q$ disks $D_{l}(\beta),$ for the $\beta$
satisfying $\pi\beta=\gamma,$ cover $D_{l+1}(\gamma).$ The result
follows.

(ii) Clearly $P\in N_{l}(\beta)$ if and only if $\pi^{nl}P(\pi^{-l}u+\pi^{-l}\beta)\in N_{0}(0).$
It is therefore enough to prove that $|P(u)|\le1$ for all $u\in\mathcal{O}_{F}$
if and only if $P\in\mathcal{O}_{E}[u]^{\le n}.$ This is well-known,
but note that it fails if $n\ge q$ (consider $\pi^{-1}(u^{q}-u)$).

(iii) This is an immediate consequence of (ii). \end{proof}

\subsection{Passing from $\mathcal{O}_{E}[B]$-modules to $\mathcal{O}_{E}[G]$-modules}

Consider the representation $V_{\rho},$ where $\rho=\tau\otimes\sigma,$
$\tau=\det(.)^{m}\otimes Sym^{n},$ and $\sigma=Ind_{B}^{G}(\chi_{1},\chi_{2})$
are as in the introduction.

\begin{proposition} Let $v_{1},\dots,v_{r}\in V_{\sigma}$ be such
that the module $\Lambda_{\sigma}=\sum_{j=1}^{r}\mathcal{O}_{E}[B]v_{j}$
spans $V_{\sigma}$ over $E.$ Let 
\begin{equation}
\Lambda=\sum_{i=0}^{n}\sum_{j=1}^{r}\mathcal{O}_{E}[B]\left(u^{i}\otimes v_{j}\right)\subset V_{\rho}.
\end{equation}
Then $\Lambda$ is commensurable with every cyclic $\mathcal{O}_{E}[G]$-submodule
of $V_{\rho}.$ \end{proposition}

\begin{proof} Let $K=GL_{2}(\mathcal{O}_{F})$ and recall that $G=BK.$
If $N\le K$ is a subgroup of finite index fixing all the $v_{j},$
then $N$ preserves the finitely generated $\mathcal{O}_{E}$-submodule
\begin{equation}
\sum_{i,j}\mathcal{O}_{E}(u^{i}\otimes v_{j}),
\end{equation}
because $\tau(K)$ preserves $\mathcal{O}_{E}[u]^{\le n}.$ It follows
that $\sum_{i,j}\mathcal{O}_{E}[K](u^{i}\otimes v_{j})$ is finitely
generated over $\mathcal{O}_{E}.$ Since $\Lambda$ spans $V_{\rho}$
over $E,$ there is a constant $c\in E$ such that 
\begin{equation}
\sum_{i,j}\mathcal{O}_{E}[K](u^{i}\otimes v_{j})\subset c\Lambda.
\end{equation}
But then 
\begin{eqnarray}
\sum_{i,j}\mathcal{O}_{E}[G](u^{i}\otimes v_{j}) & = & \mathcal{O}_{E}[B]\sum_{i,j}\mathcal{O}_{E}[K](u^{i}\otimes v_{j})\notag\\
 & \subset & \mathcal{O}_{E}[B](c\Lambda)=c\Lambda.
\end{eqnarray}
On the other hand, $\Lambda\subset\sum_{i,j}\mathcal{O}_{E}[G](u^{i}\otimes v_{j}).$
The two inclusions prove the proposition, since the sum of a finite
number of cyclic modules, all being commensurable, is again commensurable
with any cyclic module. \end{proof}

\begin{corollary} To prove Theorem 1.2 we may replace $V_{\rho}^{0}$
by $\Lambda.$ \end{corollary}

\subsection{The Kirillov model and a choice of $\Lambda$}

Assume from now on that $\chi_{1}\neq\omega\chi_{2}.$ The exceptional
case $\chi_{1}=\omega\chi_{2}$ requires special attention and will
be dealt with in the end. Let $\mathcal{K}$ be the model of $V_{\rho}$
described in the introduction. For $\left\{ v_{j}\right\} $ we choose
the two functions 
\begin{equation}
v_{1}=F_{0}^{\prime}(x)=1_{\mathcal{O}_{F}}\chi_{1},\,\,\, v_{2}=F_{0}^{\prime\prime}=1_{\mathcal{O}_{F}}\omega\chi_{2}.
\end{equation}
Let $F_{k}^{\prime}(x)=F_{0}^{\prime}(\pi^{-k}x)$ and similarly $F_{k}^{\prime\prime}(x)=F_{0}^{\prime\prime}(\pi^{-k}x).$
Since 
\begin{equation}
\sigma\left(\left(\begin{array}{ll}
\pi^{-k} & -\pi^{-k}\beta\\
 & 1
\end{array}\right)\right)F_{0}^{\prime}(x)=\psi_{\beta}(-\pi^{-k}x)F_{k}^{\prime}(x)
\end{equation}
and similarly for $F_{0}^{\prime\prime}(x),$ we see that $\Lambda_{\sigma}=\mathcal{O}_{E}[B]F_{0}^{\prime}+\mathcal{O}_{E}[B]F_{0}^{\prime\prime}$
spans $V_{\sigma}$ over $E$.

\begin{lemma} Let $\Lambda=\sum_{i=0}^{n}\sum_{j=1}^{2}\mathcal{O}_{E}[B]\left(u^{i}\otimes v_{j}\right),$
where $v_{1}=F_{0}^{\prime}$ and $v_{2}=F_{0}^{\prime\prime}.$ Then
every element of $\Lambda$ can be written as a finite sum 
\begin{equation}
\phi=\sum_{k=k_{0}}^{\infty}\sum_{\beta\in W}c_{k}^{\prime}(\beta)\psi_{\beta}(-\pi^{-k}x)F_{k}^{\prime}(x)+c_{k}^{\prime\prime}(\beta)\psi_{\beta}(-\pi^{-k}x)F_{k}^{\prime\prime}(x),
\end{equation}
where $c_{k}^{\prime}(\beta),c_{k}^{\prime\prime}(\beta)\in\pi^{-km}N_{k}(\beta).$
\end{lemma}

\begin{proof} Since the central character of $\rho$ is unitary\ (condition
(1.7)(i)), it is enough to span $\Lambda$ by matrices in the mirabolic
subgroup 
\begin{equation}
\left\{ \left(\begin{array}{ll}
a & b\\
0 & 1
\end{array}\right)\right\} \le B.
\end{equation}
Furthermore, as $B\cap K$ stabilizes $\sum_{i=0}^{n}\sum_{j=1}^{2}\mathcal{O}_{E}\left(u^{i}\otimes v_{j}\right),$
we see that 
\begin{eqnarray*}
\Lambda & = & \sum_{k\in\Bbb{Z}}\sum_{\beta\in W}\mathcal{O}_{E}\rho\left(\left(\begin{array}{ll}
\pi^{-k} & -\pi^{-k}\beta\\
 & 1
\end{array}\right)\right)\left(u^{i}\otimes v_{j}\right)\\
 & = & \sum_{k\in\Bbb{Z}}\sum_{\beta\in W}\pi^{-km}(\pi^{-k})^{n-i}(u-\pi^{-k}\beta)^{i}\otimes\psi_{\beta}(-\pi^{-k}x)\left(\mathcal{O}_{E}F_{k}^{\prime}(x)+\mathcal{O}_{E}F_{k}^{\prime\prime}(x)\right).
\end{eqnarray*}
The coefficients $(\pi^{-k})^{n-i}(u-\pi^{-k}\beta)^{i}\in N_{k}(\beta),$
see Lemma 2.2(ii). \end{proof}

\section{The unramified case}

\subsection{The recursion relations}

Assume now that $\chi_{1}$ and $\chi_{2}$ are unramified. Then 
\begin{equation}
F_{k}^{\prime}(x)=\sum_{l=k}^{\infty}\lambda^{l-k}\phi_{l},\,\,\, F_{k}^{\prime\prime}(x)=\sum_{l=k}^{\infty}\mu^{l-k}\phi_{l}.
\end{equation}

Pick a $\phi\in\Lambda.$ Substituting (3.1) in the expression (2.16),
and rearranging the sum ``by annuli'' we get 
\begin{equation}
\phi=\sum_{l=k_{0}}^{\infty}\sum_{\beta\in W}C_{l}(\beta)\psi_{\beta}(-\pi^{-l}x)\phi_{l}(x),
\end{equation}
where 
\begin{eqnarray}
C_{l}(\beta) & = & C_{l}^{\prime}(\beta)+C_{l}^{\prime\prime}(\beta),\\
C_{l}^{\prime}(\beta) & = & \sum_{k=k_{0}}^{l}\lambda^{l-k}\sum_{\pi^{l-k}\alpha=\beta}c_{k}^{\prime}(\alpha),\notag\\
C_{l}^{\prime\prime}(\beta) & = & \sum_{k=k_{0}}^{l}\mu^{l-k}\sum_{\pi^{l-k}\alpha=\beta}c_{k}^{\prime\prime}(\alpha).\notag
\end{eqnarray}
We deduce that 
\begin{eqnarray}
C_{k_{0}}^{\prime}(\beta) & = & c_{k_{0}}^{\prime}(\beta)\\
C_{l}^{\prime}(\beta) & = & \lambda\sum_{\pi\alpha=\beta}C_{l-1}^{\prime}(\alpha)+c_{l}^{\prime}(\beta),\notag
\end{eqnarray}
and similarly for $C_{l}^{\prime\prime}(\beta),$ with $\mu$ instead
of $\lambda.$ We now derive from these relations a recursion relation
for the $C_{l}(\beta),$ going two generations backwards.

\begin{lemma} Let $c_{l}=c_{l}^{\prime}+c_{l}^{\prime\prime}.$ Then
$C_{k_{0}}(\beta)=c_{k_{0}}(\beta)$ and 
\begin{eqnarray}
C_{l+1}(\gamma) & = & (\lambda+\mu)\sum_{\pi\beta=\gamma}C_{l}(\beta)-\mu\lambda\sum_{\pi\beta=\gamma}\sum_{\pi\alpha=\beta}C_{l-1}(\alpha)\notag\\
 &  & -\sum_{\pi\beta=\gamma}(\lambda c_{l}^{\prime\prime}(\beta)+\mu c_{l}^{\prime}(\beta))+c_{l+1}(\gamma).
\end{eqnarray}
\end{lemma}

\begin{proof} We add the relations that we have obtained for $C_{l}^{\prime}(\beta)$
and $C_{l}^{\prime\prime}(\beta)$ and rearrange them. We do the same
at level $l+1$. Letting $\alpha,$ $\beta$ and $\gamma$ range over
$W$ as usual, we get 
\begin{eqnarray}
C_{l}(\beta) & = & \lambda\sum_{\pi\alpha=\beta}C_{l-1}(\alpha)+(\mu-\lambda)\sum_{\pi\alpha=\beta}C_{l-1}^{\prime\prime}(\alpha)+c_{l}(\beta),\notag\\
C_{l+1}(\gamma) & = & \lambda\sum_{\pi\beta=\gamma}C_{l}(\beta)+(\mu-\lambda)\sum_{\pi\beta=\gamma}C_{l}^{\prime\prime}(\beta)+c_{l+1}(\gamma).
\end{eqnarray}
To deal with the middle term in the \emph{second} equation we use
the recursive relation for $C_{l}^{\prime\prime}(\beta)$ and then
eliminate $(\mu-\lambda)\sum_{\pi\alpha=\beta}C_{l-1}^{\prime\prime}(\alpha)$
using the \emph{first }equation: 
\begin{eqnarray}
(\mu-\lambda)\sum_{\pi\beta=\gamma}C_{l}^{\prime\prime}(\beta) & = & (\mu-\lambda)\sum_{\pi\beta=\gamma}\left(\mu\sum_{\pi\alpha=\beta}C_{l-1}^{\prime\prime}(\alpha)+c_{l}^{\prime\prime}(\beta)\right)\notag\\
 & = & \mu\sum_{\pi\beta=\gamma}\left(C_{l}(\beta)-\lambda\sum_{\pi\alpha=\beta}C_{l-1}(\alpha)-c_{l}(\beta)\right)\notag\\
 &  & +(\mu-\lambda)\sum_{\pi\beta=\gamma}c_{l}^{\prime\prime}(\beta)\notag\\
 & = & \mu\sum_{\pi\beta=\gamma}C_{l}(\beta)-\mu\lambda\sum_{\pi\beta=\gamma}\sum_{\pi\alpha=\beta}C_{l-1}(\alpha)\notag\\
 &  & -\sum_{\pi\beta=\gamma}(\lambda c_{l}^{\prime\prime}(\beta)+\mu c_{l}^{\prime}(\beta)).
\end{eqnarray}
The lemma follows from this. \end{proof}

\subsection{Conclusion of the proof}

Let $\rho$ satisfy the conditions of Thoerem 1.2, i.e. the estimates
(1.7)(i) and (ii) on $\lambda$ and $\mu,$ and $n<q.$ Pick a $\phi\in\Lambda$
as before, and expand it as in (3.2). Assume that it vanishes outside
of $\mathcal{O}_{F}.$ Let 
\begin{equation}
M_{l}(\beta)=q^{-1}\pi^{-n-lm}N_{l}(\beta).
\end{equation}

\begin{lemma} For every $k_{0}\le l\le0$ and every $\beta\in W,$
$C_{l}(\beta)\in M_{l}(\beta).$ \end{lemma}

\begin{proof} We apply Lemma 2.2 and Lemma 2.5, and prove the desired
bound on $C_{l}(\beta)$ by increasing induction on $l.$

When $l=k_{0},$ $C_{k_{0}}(\beta)=c_{k_{0}}(\beta)\in\pi^{-k_{0}m}N_{k_{0}}(\beta)\subset M_{k_{0}}(\beta).$
Suppose that the lemma has been established up to index $l,$ and
$l+1\le0.$ Then $C_{l}(\beta)$ (resp. $C_{l-1}(\alpha)$) depends
only on $\pi\beta$ (resp. $\pi\alpha$), since $\phi$ vanishes on
$F-\mathcal{O}_{F}.$ We invoke the recursion relation (3.5) for $C_{l+1}(\gamma).$
The term 
\begin{equation}
\sum_{\pi\beta=\gamma}(\lambda c_{l}^{\prime\prime}(\beta)+\mu c_{l}^{\prime}(\beta))\in M_{l+1}(\gamma)
\end{equation}
since $c_{l}^{\prime}(\beta),c_{l}^{\prime\prime}(\beta)\in\pi^{-lm}N_{l}(\beta)$,
$|\mu|,|\lambda|\le|q^{-1}\pi^{-n-m}|,$ and because of the relation
$N_{l}(\beta)\subset N_{l+1}(\gamma),$ that holds whenever $\pi\beta=\gamma.$
That 
\begin{equation}
c_{l+1}(\gamma)\in M_{l+1}(\gamma)
\end{equation}
is clear. The term 
\begin{equation}
(\lambda+\mu)\sum_{\pi\beta=\gamma}C_{l}(\beta)\in M_{l+1}(\gamma)
\end{equation}
because the $q$ summands $C_{l}(\beta)$ are \emph{equal}, hence
belong to 
\begin{equation}
\bigcap_{\pi\beta=\gamma}M_{l}(\beta)=q^{-1}\pi^{-n-lm}\bigcap_{\pi\beta=\gamma}N_{l}(\beta)=q^{-1}\pi^{-lm}N_{l+1}(\gamma).
\end{equation}
Thus $\sum_{\pi\beta=\gamma}C_{l}(\beta)\in$ $\pi^{-lm}N_{l+1}(\gamma),$
while $|\lambda+\mu|\le|q^{-1}\pi^{-n-m}|.$ Finally, 
\begin{equation}
\mu\lambda\sum_{\pi\beta=\gamma}\sum_{\pi\alpha=\beta}C_{l-1}(\alpha)\in M_{l+1}(\gamma)
\end{equation}
for similar reasons: For a given $\beta,$ the $q$ summands $C_{l-1}(\alpha)$
are equal, so belong to 
\begin{equation}
\bigcap_{\pi\alpha=\beta}M_{l-1}(\alpha)=q^{-1}\pi^{-n-(l-1)m}\bigcap_{\pi\alpha=\beta}N_{l-1}(\alpha)=q^{-1}\pi^{-(l-1)m}N_{l}(\beta).
\end{equation}
This implies that their sum, $\sum_{\pi\alpha=\beta}C_{l-1}(\alpha)\in\pi^{-(l-1)m}N_{l}(\beta)\subset\pi^{-(l-1)m}N_{l+1}(\gamma).$
But $|\mu\lambda|=|q^{-1}\pi^{-n-2m}|,$ so for every $\beta,$
\begin{equation}
\mu\lambda\sum_{\pi\alpha=\beta}C_{l-1}(\alpha)\in q^{-1}\pi^{-n-(l+1)m}N_{l+1}(\gamma)=M_{l+1}(\gamma).
\end{equation}

Since each of the four terms in (3.6) has been shown to lie in $M_{l+1}(\gamma)$,
the proof of the induction step is complete. \end{proof}

When $l=0,$ $C_{0}(\beta)\in M_{0}(\beta),$ and this proves Theorem
1.2.

\section{The tamely ramified case}

For the sake of completeness we treat also case (2) of the theorem,
which is covered by {[}K-dS12{]}. The proof is the same, except that
we have cleaned up the computations.

\subsection{The recursion relations}

Assume from now on that at least one of the characters $\chi_{1}$
and $\chi_{2}$ is ramified, but $\tau=\det(.)^{m},$ i.e. $n=0.$
Since a twist of $\rho$ by a character of finite order does not affect
the validity of Theorem 1.2, we may assume that $\chi_{2}$ is unramified.
We let $\varepsilon$ be the restriction of $\chi_{1}$ to $U_{F},$
and extend it to a character of $F^{\times}$ so that $\varepsilon(\pi)=1.$
We denote by $\nu\ge1$ the conductor of $\varepsilon.$ Letting $\lambda=\chi_{1}(\pi)$
and $\mu=\chi_{2}(\pi)$ as before, we have 
\begin{equation}
\chi_{1}(u\pi^{k})=\varepsilon(u)\lambda^{k},\,\chi_{2}(u\pi^{k})=\mu^{k}
\end{equation}
if $u\in U_{F}.$

Recall that 
\begin{equation}
F_{k}^{\prime}=\varepsilon\sum_{l=k}^{\infty}\lambda^{l-k}\phi_{l},\,\,\, F_{k}^{\prime\prime}=\sum_{l=k}^{\infty}\mu^{l-k}\phi_{l}.
\end{equation}
The module $\Lambda$ consists this time of functions of the form
\begin{eqnarray}
\phi(x) & = & \sum_{k=k_{0}}^{\infty}\sum_{\beta\in W}c_{k}^{\prime}(\beta)\psi_{\beta}(-\pi^{-k}x)F_{k}^{\prime}(x)+c_{k}^{\prime\prime}(\beta)\psi_{\beta}(-\pi^{-k}x)F_{k}^{\prime\prime}(x)\\
 & = & \sum_{l=k_{0}}^{\infty}\sum_{\beta\in W}C_{l}(\beta)\psi_{\beta}(-\pi^{-l}x)\phi_{l}(x),\notag
\end{eqnarray}
with $c_{k}^{\prime}(\beta),c_{k}^{\prime\prime}(\beta)\in\pi^{-mk}\mathcal{O}_{E},$
and some $C_{l}(\beta)$ which we are now going to compute. Let, as
before 
\begin{eqnarray}
C_{l}^{\prime}(\beta) & = & \sum_{k=k_{0}}^{l}\lambda^{l-k}\sum_{\pi^{l-k}\alpha=\beta}c_{k}^{\prime}(\alpha)\notag\\
C_{l}^{\prime\prime}(\beta) & = & \sum_{k=k_{0}}^{l}\mu^{l-k}\sum_{\pi^{l-k}\alpha=\beta}c_{k}^{\prime\prime}(\alpha).
\end{eqnarray}
These coefficients satisfy the \emph{recursion relations} 
\begin{eqnarray}
C_{k_{0}}^{\prime}(\beta) & = & c_{k_{0}}^{\prime}(\beta)\\
C_{l}^{\prime}(\beta) & = & \lambda\sum_{\pi\alpha=\beta}C_{l-1}^{\prime}(\alpha)+c_{l}^{\prime}(\beta),\notag
\end{eqnarray}
and similarly for $C_{l}^{\prime\prime}(\beta),$ with $\mu$ instead
of $\lambda.$ In terms of the $C_{l}^{\prime}(\beta)$ and the $C_{l}^{\prime\prime}(\beta)$
we have 
\begin{equation}
\phi(x)=\varepsilon(x)\sum_{l=k_{0}}^{\infty}\sum_{\beta\in W}C_{l}^{\prime}(\beta)\psi_{\beta}(-\pi^{-l}x)\phi_{l}(x)+\sum_{l=k_{0}}^{\infty}\sum_{\beta\in W}C_{l}^{\prime\prime}(\beta)\psi_{\beta}(-\pi^{-l}x)\phi_{l}(x).
\end{equation}

Invoking the Fourier expansion of $\varepsilon(x)\phi_{l}(x)$ (see
{[}K-dS12{]}, Corollary 2.2) we finally get the formula 
\begin{equation}
C_{l}(\beta)=\frac{\tau(\varepsilon^{-1})}{q^{\nu}}\sum_{u\in U_{F}/U_{F}^{\nu}}\varepsilon^{-1}(u)C_{l}^{\prime}(\beta-\pi^{-\nu}u)+C_{l}^{\prime\prime}(\beta).
\end{equation}
Here $U_{F}^{\nu}$ denotes the group of units which are congruent
to 1 modulo $\pi^{\nu},$ and $\tau(\varepsilon^{-1})$ is the Gauss
sum 
\begin{equation}
\tau(\varepsilon^{-1})=\sum_{u\in U_{F}/U_{F}^{\nu}}\psi(\pi^{-\nu}u)\varepsilon(u).
\end{equation}

We recall the well-known identity 
\begin{equation}
\tau(\varepsilon)\tau(\varepsilon^{-1})=\varepsilon(-1)q^{\nu}.
\end{equation}

\subsection{Operators on functions on $W$}

As in {[}K-dS12{]}, Section 3.4, we introduce some operators on the
space $\mathcal{C}$ of $E$-valued functions on $W$ with finite
support. If $f\in\mathcal{C}$ we define
\begin{itemize}
\item The $\emph{suspension}$ of $f$
\begin{equation}
Sf(\beta)=\sum_{\pi\alpha=\beta}f(\alpha).
\end{equation}

\item The \emph{convolution} of $f$ with a character $\xi$ of $U_{F},$
of conductor $\nu\ge1$
\begin{equation}
E_{\xi}f(\beta)=\frac{\tau(\xi^{-1})}{q^{\nu}}\sum_{u\in U_{F}/U_{F}^{\nu}}\xi^{-1}(u)f(\beta-\pi^{-\nu}u).
\end{equation}

\item The operator $\Pi$
\begin{equation}
\Pi f(\beta)=f(\pi\beta).
\end{equation}

\end{itemize}
We decompose $\mathcal{C}$ as a direct sum $\mathcal{C}=\mathcal{C}_{0}\bigoplus\mathcal{C}_{1},$
where 
\begin{eqnarray}
\mathcal{C}_{0} & = & \left\{ f|\,\forall\beta,\,\,\sum_{\pi t=0}f(\beta+t)=0\right\} \\
\mathcal{C}_{1} & = & \left\{ f|\, f(\beta)\text{ depends only on }\pi\beta\right\} .\notag
\end{eqnarray}

\begin{lemma} (i) The projection onto $\mathcal{C}_{1}$ is 
\begin{equation}
P_{1}=\frac{1}{q}\Pi S.
\end{equation}

(ii) Let $\xi$ be any non-trivial character. Then the projection
onto $\mathcal{C}_{0}$ is 
\begin{equation}
P_{0}=E_{\xi}E_{\xi^{-1}}=E_{\xi^{-1}}E_{\xi}.
\end{equation}

(iii) If $\xi$ is non-trivial then $SE_{\xi}=0$ and $E_{\xi}E_{\xi^{-1}}E_{\xi}=E_{\xi}.$
\end{lemma}

\begin{proof} All the statements are elementary, and best understood
if we associate to $f$ its Fourier transform 
\begin{equation}
\hat{f}(x)=\sum_{\beta\in W}f(\beta)\psi_{\beta}(x)
\end{equation}
($x\in\mathcal{O}_{F}$) and apply Lemma 2.1. See {[}K-dS12{]}, Section
3.4. \end{proof}

For $f,g_{1},\dots,g_{r}\in\mathcal{C}$ we write $f=O(g_{1},\dots,g_{r})$
to mean that in the $\sup$ norm $||f||\le\max||g_{i}||.$

\subsection{Conclusion of the proof in the tamely ramified case}

We assume from now on that $\nu=1,$ i.e. $\varepsilon$ is tamely
ramified. The Breuil-Schneider estimates on $\lambda$ and $\mu$
are 
\begin{eqnarray*}
|\pi^{-m}| & \le & |\lambda|,|\mu|\le|q^{-1}\pi^{-m}|\\
\,\,|\lambda\mu| & = & |q^{-1}\pi^{-2m}|.
\end{eqnarray*}

Fix a $\phi\in\Lambda$ as in (4.3), so that 
\begin{equation}
c_{k}^{\prime},\, c_{k}^{\prime\prime}=O(\pi^{-mk}),
\end{equation}
and assume that it vanishes off $\mathcal{O}_{F}.$ We shall prove
by increasing induction on $l$ that for $l\le0$
\begin{equation}
C_{l}^{\prime},\,\, C_{l}^{\prime\prime}=O(q^{-1}\pi^{-ml}).
\end{equation}
When we reach $l=0$ this will imply Theorem 1.2, even uniformly in
$\beta$, thanks to the fact that the algebraic part of $\rho$ is
essentially trivial.

Using the notation of the last sub-section, we can write the recursion
relations (4.5) as 
\begin{eqnarray}
C_{k_{0}}^{\prime} & = & c_{k_{0}}^{\prime},\,\, C_{k_{0}}^{\prime\prime}=c_{k_{0}}^{\prime\prime}\notag\\
C_{l}^{\prime} & = & \lambda SC_{l-1}^{\prime}+c_{l}^{\prime}\\
C_{l}^{\prime\prime} & = & \mu SC_{l-1}^{\prime\prime}+c_{l}^{\prime\prime}.\notag
\end{eqnarray}
Besides $C_{l}(\beta)$ we introduce $\tilde{C}_{l}(\beta)$ so that
the following formulae hold 
\begin{eqnarray}
C_{l} & = & E_{\varepsilon}C_{l}^{\prime}+C_{l}^{\prime\prime}\\
\tilde{C}_{l} & = & E_{\varepsilon^{-1}}C_{l}^{\prime\prime}+C_{l}^{\prime}.\notag
\end{eqnarray}
Here the first formula is just (4.7). The second shows that the amplitudes
$\tilde{C}_{l}(\beta)$ are analogously associated with the function
$\tilde{\phi}(x)=\varepsilon^{-1}(x)\phi(x).$

Next, we observe that since $SE_{\varepsilon}=SE_{\varepsilon^{-1}}=0,$
we can rewrite the recursion relations as 
\begin{eqnarray}
C_{l}^{\prime} & = & \lambda S\tilde{C}_{l-1}+c_{l}^{\prime}\notag\\
C_{l}^{\prime\prime} & = & \mu SC_{l-1}+c_{l}^{\prime\prime}.
\end{eqnarray}
For $l\le0$ the functions $C_{l-1}$ and $\tilde{C}_{l-1}$ belong
to the subspace that we have called $\mathcal{C}_{1},$ because $\phi$
and $\tilde{\phi}$ vanish on $\pi^{l-1}U_{F}$. This implies the
following result.

\begin{lemma} For $l\le0,$
\begin{eqnarray}
C_{l}^{\prime} & = & O(\lambda q\tilde{C}_{l-1},c_{l}^{\prime})\notag\\
C_{l}^{\prime\prime} & = & O(\mu qC_{l-1},c_{l}^{\prime\prime}).
\end{eqnarray}
\end{lemma}

We can now proceed with the induction. When $l=k_{0}$ (4.17) clearly
implies (4.18). Assume that $l\le0$ and that (4.18) has been established
up to index $l-1.$ As $C_{l-2}^{\prime}=O(q^{-1}\pi^{-m(l-2)}),$
and as $C_{l-2}^{\prime\prime}=O(q^{-1}\pi^{-m(l-2)})=O(q^{-2}\tau(\varepsilon^{-1})\pi^{-m(l-2)}),$
we obtain from (4.20) and the fact that $\nu=1$ the estimate 
\begin{equation}
C_{l-2}=O(q^{-2}\tau(\varepsilon^{-1})\pi^{-m(l-2)}).
\end{equation}
By the lemma, this gives 
\begin{equation}
C_{l-1}^{\prime\prime}=O(\mu q^{-1}\tau(\varepsilon^{-1})\pi^{-m(l-2)},c_{l-1}^{\prime\prime})=O(\mu q^{-1}\tau(\varepsilon^{-1})\pi^{-m(l-2)})
\end{equation}
(the last equality coming from $|\mu q^{-1}\tau(\varepsilon^{-1})|\ge|\pi^{-m}|$).
A second application of (4.20), the identity (4.9), and the induction
hypothesis for $C_{l-1}^{\prime}$ (recall $|\mu|\ge|\pi^{-m}|$)
yield 
\begin{equation}
\tilde{C}_{l-1}=O(\mu q^{-1}\pi^{-m(l-2)}).
\end{equation}
A second application of the lemma finally gives 
\begin{eqnarray}
C_{l}^{\prime} & = & O(\lambda\mu\pi^{-m(l-2)},c_{l}^{\prime\prime})\notag\\
 & = & O(q^{-1}\pi^{-ml},c_{l}^{\prime\prime})=O(q^{-1}\pi^{-ml}).
\end{eqnarray}
Symmetrically, we get the same estimate on $C_{l}^{\prime\prime}.$
This completes the proof of (4.18) at level $l$, and with it, the
proof of Theorem 1.2.

\section{The case $\chi_{1}=\omega\chi_{2}$}

We finally deal with the one excluded case, when $\chi_{1}=\omega\chi_{2}$.
After a twist by a character of finite order we may assume that $\chi_{1}$
is unramified. In this case $\lambda=\mu$ and the Kirillov model
is the space 
\begin{equation}
\mathcal{K}=C_{c}^{\infty}(F,\tau)\chi_{1}+C_{c}^{\infty}(F,\tau)v\chi_{1},
\end{equation}
where $v:F^{\times}\rightarrow\Bbb{Z}\subset E$ is the normalized
valuation. The action of $B$ is still given by (1.10). Once more,
$\mathcal{K}$ contains $\mathcal{K}_{0}=C_{c}^{\infty}(F^{\times},\tau)$
as a subspace. When $\tau=1,$ the quotient $\mathcal{K}/\mathcal{K}_{0}$
is the Jacquet module. The torus acts on it non-semisimply, by 
\begin{equation}
\left(\begin{array}{ll}
t_{1}\\
 & t_{2}
\end{array}\right)\mapsto\chi_{1}(t_{1}t_{2})\left(\begin{array}{ll}
1 & v(t_{1}/t_{2})\\
 & 1
\end{array}\right).
\end{equation}

Following the notation of Section 3, we let 
\begin{equation}
F_{0}^{\prime}=\chi_{1}1_{\mathcal{O}_{F}},\,\,\, F_{0}^{\prime\prime}=-v\chi_{1}1_{\mathcal{O}_{F}}
\end{equation}
and 
\begin{equation}
F_{k}^{\prime}=\sum_{l=k}^{\infty}\lambda^{l-k}\phi_{l},\,\,\,\, F_{k}^{\prime\prime}=\sum_{l=k}^{\infty}(k-l)\lambda^{l-k}\phi_{l}.
\end{equation}
The module $\Lambda$ consists of all the functions $\phi$ as in
(2.16), and any such $\phi$ can be expanded ``by annuli'' as in
(3.2). The coefficients of the expansion are given by (3.3), except
that the last equation now takes the shape 
\begin{equation}
C_{l}^{\prime\prime}(\beta)=\sum_{k=k_{0}}^{l}(k-l)\lambda^{l-k}\sum_{\pi^{l-k}\alpha=\beta}c_{k}^{\prime\prime}(\alpha).
\end{equation}
The recursion relation for $C_{l}^{\prime}(\beta)$ is given by (3.4)
but $C_{l}^{\prime\prime}(\beta)$ needs a modification.

\begin{lemma} We have $C_{k_{0}}^{\prime\prime}(\beta)=0,$ $C_{k_{0}+1}^{\prime\prime}(\beta)=-\lambda\sum_{\pi\alpha=\beta}c_{k_{0}}^{\prime\prime}(\alpha)$,
and for $l>k_{0}$ 
\begin{equation}
C_{l+1}^{\prime\prime}(\gamma)=2\lambda\sum_{\pi\beta=\gamma}C_{l}^{\prime\prime}(\beta)-\lambda^{2}\sum_{\pi\beta=\gamma}\sum_{\pi\alpha=\beta}C_{l-1}^{\prime\prime}(\alpha)-\lambda\sum_{\pi\beta=\gamma}c_{l}^{\prime\prime}(\beta).
\end{equation}
\end{lemma}

\begin{proof} A straightforward exercise. \end{proof}

\begin{lemma} The following recursion relation holds: 
\begin{equation}
C_{l+1}(\gamma)=2\lambda\sum_{\pi\beta=\gamma}C_{l}(\beta)-\lambda^{2}\sum_{\pi\beta=\gamma}\sum_{\pi\alpha=\beta}C_{l-1}(\alpha)-\lambda\sum_{\pi\beta=\gamma}(c_{l}^{\prime\prime}(\beta)+c_{l}^{\prime}(\beta))+c_{l+1}^{\prime}(\gamma).
\end{equation}
\end{lemma}

\begin{proof} We write 
\begin{eqnarray}
C_{l+1}^{\prime}(\gamma) & = & \lambda\sum_{\pi\beta=\gamma}C_{l}^{\prime}(\beta)+c_{l+1}^{\prime}(\gamma)\notag\\
 & = & 2\lambda\sum_{\pi\beta=\gamma}C_{l}^{\prime}(\beta)-\lambda\sum_{\pi\beta=\gamma}\left(\lambda\sum_{\pi\alpha=\beta}C_{l-1}^{\prime}(\alpha)+c_{l}^{\prime}(\beta)\right)+c_{l+1}^{\prime}(\gamma)\notag\\
 & = & 2\lambda\sum_{\pi\beta=\gamma}C_{l}^{\prime}(\beta)-\lambda^{2}\sum_{\pi\beta=\gamma}\sum_{\pi\alpha=\beta}C_{l-1}^{\prime}(\alpha)-\lambda\sum_{\pi\beta=\gamma}c_{l}^{\prime}(\beta)+c_{l+1}^{\prime}(\gamma)
\end{eqnarray}
and we add the result to the recursive relation for $C_{l+1}^{\prime\prime}(\gamma).$
\end{proof}

Note the similarity with Lemma 3.1. The rest of the proof of Theorem
1.2 is now identical to that given in the case $\lambda\neq\mu$ in
Section 3.2.

\bigskip{}

\textbf{References\medskip{}
}

{[}B-B-C{]} L.Berger, C.Breuil, P.Colmez (Eds.): \emph{Représentations
}$p$\emph{-adiques de groupes} $p$\emph{-adiques,} Astérisque, \textbf{319,
330, 331 }(2008-2010).\medskip{}

{[}Br03{]} C. Breuil: \emph{Sur quelques représentations modulaires
et} $p$\emph{-adiques de} $GL_{2}(\Bbb{Q}_{p})$ \emph{II}, J. Inst.
Math. Jussieu \textbf{2} (2003), 1-36.\medskip{}

{[}Br04{]} C.Breuil, \emph{Invariant }$\mathcal{L}$ \emph{et série
spéciale }$p$\emph{-adique}, Annales Scientifiques E.N.S. \textbf{37}
(2004), 559-610.\medskip{}

{[}Br10{]} C.Breuil: \emph{The emerging p-adic Langlands programme,\,}Proceedings
of the International Congress of Mathematicians, Hyderabad, India,
2010.\medskip{}

{[}Br-Sch07{]} C.Breuil, P.Schneider: \emph{First steps towards }$p\emph{-adic}$\emph{\ Langlands
functoriality,} J. Reine Angew. Math. \textbf{610} (2007)149-180.\medskip{}

{[}Bu98{]} D.Bump: \emph{Automorphic Forms and Representations}, Cambridge
University Press, 1998.\medskip{}

{[}dI12{]} M.de Ieso, \emph{Analyse} $p$\emph{-adique et complétés
unitaires universeles pour} $GL_{2}(F),$ Ph.D. thesis, Orsay, 2012.\medskip{}

{[}dS08{]} E.de Shalit: \emph{Integral structures in locally algebraic
representations,} unpublished notes (2008).\medskip{}

{[}Hu09{]} Y.Hu: \emph{Normes invariantes et existence de filtrations
admissibles}, J. Reine Angew. Math. \textbf{634} (2009), 107-141.\medskip{}

{[}K-dS12{]} D.Kazhdan, E.de Shalit: \emph{Kirillov models and integral
structures in }$p$\emph{-adic smooth representations of} $GL_{2}(F),$
J.Algebra \textbf{353} (2012), 212-223.\medskip{}

{[}P01{]} D.Prasad, \emph{appendix} to P.Schneider, J.Teitelbaum:
$U(\frak{g})$\emph{-finite locally analytic representations, }Representation
Theory \textbf{5} (2001), 111-128.\medskip{}

{[}So13{]} C.Sorensen: \emph{A proof of the Breuil-Schneider conjecture
in the indecomposable case}, Annals of Mathematics \textbf{177} (2013),
1-16.\medskip{}

{[}T93{]} J.Teitelbaum, \emph{Modular representations of }$PGL_{2}$\emph{\ and
automorphic forms for Shimura curves}, Invent. Math. \textbf{113}
(1993), 561-580.\medskip{}

{[}V08{]} M.-F.Vigneras, \emph{A criterion for integral structures
and coefficient systems on the tree of }$PGL(2,F),$ Pure and Appl.
Math. Quat. \textbf{4 }(2008), 1291-1316.
\end{document}